\documentclass[12pt]{amsart}
\usepackage{latexsym,fancyhdr,amssymb,color,amsmath,amsthm,graphicx,listings,comment}
\usepackage[section]{placeins} 
\newtheorem{thm}{Theorem} \newtheorem{lemma}{Lemma}  \newtheorem{coro}{Corollary}
\let\paragraph\subsection
\title{A simple sphere theorem for graphs}
\author{Oliver Knill}
\date{October 6, 2019}
\address{Department of Mathematics \\ Harvard University \\ Cambridge, MA, 02138 }

\begin{document}
\maketitle

\begin{abstract}
A finite simple graph $G$ is declared to have positive curvature if 
every in $G$ embedded wheel graph has five or six vertices. A d-graph 
is a finite simple graph $G$ for which every unit sphere is a $(d-1)$-sphere.
A $d$-sphere is a $d$-graph $G$ for which there exists a vertex $x$ such that $G-x$ is contractible.
A graph $G$ is contractible if there is a vertex $x$ such that $S(x)$ and 
$G-x$ are contractible. The empty graph $0$ is the $(-1)$-sphere. The $1$-point graph $1$ is contractible.
The theorem is that for $d \geq 1$, every connected positive curvature $d$-graph
is a $d$-sphere. A discrete Synge result follows: 
a positive curvature graph is simply connected and orientable. For every $d>1$,
there are only finitely many positive curvature graphs. There are six for $d=2$
and all have diameter $\leq 3$. To prove the theorem, we use
a ``geomag lemma" which shows that every geodesic in $G$ can be extended to an immersed 
$2$-graph $S$ of positive curvature and must so be a $2$-sphere with positive curvature. 
As none of these has diameter larger than $3$, also $G$ has a diameter $3$ or less. 
This can be used to show that $G-x$ is contractible and so must be a sphere. 
\end{abstract}

\section{The result} 

\paragraph{}
A finite simple graph $G=(V,E)$ is called a {\bf $d$-graph}, if for all $x \in V$, the unit sphere $S(x)$ 
(the graph generated by the vertices attached to $x$) is a $(d-1)$-sphere. A graph $G$ is called a
{\bf $d$-sphere} if it is a $d$-graph and removing one vertex $x$ renders 
$G-x$ (the graph $G$ with $x$ and all connections to $x$ removed) is contractible. 
A graph $G$ is called {\bf contractible} if there exists a vertex $x$ so that $S(x)$ and $G-x$ are 
both contractible. These inductive definitions define d-graphs, which are discrete manifolds or 
d-spheres which are discrete spheres.
The assumptions are primed by the assumption that the empty graph $0$ is a $(-1)$-sphere 
and that the $1$-point graph $1$ is contractible. A wheel graph $S$, the unit ball of a point
in a 2-graph, is called {\bf embedded} in $G$ if it is a sub-graph of $G$ and if its vertex set 
$W$ generates $S$. In other words, $S$ is embedded if every simplex in $G$ built from vertices in
$W$ is a simples of $S$. A $d$-graph is declared to have {\bf positive curvature}, if all embedded 
wheel graphs $S$ have less than $6$ boundary points. 
This means that the wheel graph $S$ itself including the center has $5$ or $6$ vertices. 
These rather strong curvature assumption allow for a rather strong conclusion: 

\begin{thm}[Simple sphere theorem]
For $d \geq 1$, every connected positive curvature $d$-graph $G$ 
is a $d$-sphere.
\end{thm}

\paragraph{}
One could also call it the {\bf Mickey Mouse sphere theorem}.
According to \cite{StanleyUpper}, Raoul Bott once asked Richard Stanley why he 
wanted to work on a ``Mickey Mouse subject". Bott obviously was teasing as
he showed also great respect for Rota-style combinatorics (still according to \cite{StanleyUpper}).
A reference to Mickey Mouse appears also in \cite{Thurston2002} but in the context of 
hyperbolic surfaces, which can have the shape of a mouse. The name is fitting in 
the positive curvature case, as there are very few graphs which have positive 
curvature and they are all very small. 

\paragraph{}
We need $d>0$ as for $d=0$, the connectedness condition gives the $1$-point graph $G=1$, 
which is not a $0$-sphere. For $d=1$ also, the curvature condition is mute, but all connected 
$1$-graphs are $1$-spheres. For $d=2$, the curvature of a positive curvature graph 
is $K(x)=1-{\rm deg}(x)/6 \in \{1/3,1/6\}$. 
The largest positive curvature graph, the icosahedron with constant curvature $K(x)=1-5/6=1/6$, has diameter $3$. 
The smallest positive curvature graph, the octahedron with constant curvature $K(x)=1-4/6=1/3$, has diameter $2$. 
As we will see, there are exactly six positive curvature graphs in dimension $2$. 

\paragraph{}
In higher dimensions, we argue with a {\bf geomag lemma}:
any geodesic arc $C$ between two points $A,B$ can be extended 
to an embedded $2$-sphere $S$. This surface is by no means unique in general. 
It might surprise that no orientation assumption as in the continuum is needed.
It turns out that the strong curvature condition does not allow for enough room to produce a 
discrete projective plane. As the completion of the arc $C$ to a $2$-dimensional 
discrete surface has diameter $\leq 3$, also $G$ has diameter $\leq 3$. 
As $2$-dimensional spheres are simply connected, there are no shortest geodesic 
curves which are not homotopic to a point. Actually, a half sphere containing 
the closed geodesic defines the homotopy deformation, as it collapses a closed 
loop on the equator to a point on the pole.  This is a discrete Synge theorem. 
But the simple sphere theorem for graphs holds without orientability assumption. 

\paragraph{}
To see that $G$ must be a $d$-sphere, we first note that in a positive 
curvature graph, the union of unit balls centered at a unit ball is a ball. In other words, 
the set of vertices in distance $\leq 2$ form a ball. (This statement 
can fail if the positive curvature assumption is dropped, as then $B_2(x)$ can 
become non-simply connected already). Now pick a point and take $B_2(x)$, the graph generated by all vertices in 
distance $2$ or less from $x$. This is either a ball or the entire $G$. In the later
case, just remove one vertex $z$ with maximal distance to $x$ so that $B$ has a sphere 
boundary and $B$ is contractible. In the former case, cover every vertex $y \in B_2(x)$ with a 
ball $B_x(y)$ in $G$ but always avoid a fixed vertex $z$. We have now covered $G-x$ 
in a way to see that it is contractible. By definition, $G$ is then a $d$-sphere. 

\paragraph{}
Negative curvature graphs can be defined similarly. But the definition shows 
more limitations there: we have so far not seen any negative curvature
graphs, if {\bf negative curvature} means that all
embedded wheel graphs in $G$ have more than $6$ boundary points. One can 
exclude them easily in dimensions larger than $2$: \\

{\bf Remark.} There are no negative curvature graphs for $d>2$.  \\

\begin{proof} 
For $d>2$, any collection of $(d-2)$ intersections of neighboring unit spheres produces a $2$-sphere $S$.
As $S$ has Euler characteristic $2$, Gauss-Bonnet leads to some positive curvature and so at least $6$ and 
maximally $12$ wheel graphs with less than $6$ vertices. For a $3$-graph for example, the unit 
spheres are 2-spheres which must contain at least 6 positive curvature wheel graphs. For a $4$-graph, a 
discrete analogue of a $4$-manifold, the unit spheres are $3$-spheres which by assumption have to have positive 
curvature too. As we have established already, this $3$-graph contains then $2$-spheres and so some
positive curvature. 
\end{proof} 

\paragraph{}
What remains to be analyzed is the case $d=2$. There are a priori only finitely many connected 
negative curvature $2$-graphs $G$ with a given genus $g>1$
because the genus defines the Euler characteristic $\chi(G) = 2-2g<0$. As the curvature $K(x)$ of
every vertex $x \in V$ of a negative curvature graph $G=(V,E)$ is $\leq -1/6$, we see from the 
Gauss-Bonnet formula $\chi(G) = \sum_{x \in V} K(x)$ that the number of vertices in $V$ is 
bounded above by $|\chi(G)|*6 = (2-2g)*6 = (12-12g)$. 

\paragraph{}
We currently believe it should be not too difficult to prove that there are no negative curvature graphs
in the case $d=2$ too. The first case is genus $g=2$, in which case $\chi(G)=-2$. As the negative curvature 
closest to $0$ is $-1/6$, the number of vertices must be $12$ or less. So, the question is whether there 
is a $2$-graph with $12$ vertices and negative curvature. In that particular case $g=2$, there is none.
In general, it appears that there not enough vertices to generate the $g$ ``holes" needed. 
We have not a formal proof of this statement for larger $g$ but believe it can
go along similar lines as in the case $g=2$: if we look at the wheel graph centered at some vertex, 
then we already use $8$ vertices. 
Only 4 vertices are left to build a 2-graph. But they are already needed to satisfy the degree 7 requirement.
Having no vertices (magnetic balls) any more to continue building, we have to identify boundary points. 
Any such identification produces loops of length $2$ which  is not compatible with having a $2$-graph 
(every unit circle must be a circular graph with 4 or more vertices).

\section{Two-dimensional graphs}

\paragraph{}
For $d=2$, we know that the curvatures are constant in the octahedron (curvature is constant $1/3$)
and icosahedron case (curvature is constant $1/6$) and that:

\begin{lemma}[Positive curvature 2-graphs]
There are exactly six positive curvature 2-graphs which are connected.
\footnote{"Arithmetic is being able to count up to twenty 
without taking off your shoes." -- Mickey Mouse}
\end{lemma}

\begin{proof}
Because the curvature is $\geq 1/6$ and the Euler characteristic can not be larger than $2$ for a 
connected $2$-graph,
the vertex cardinalities $v$ have a priori to be in the set $\{6,7,8,9,10,11,12 \}$. 
Given such a vertex cardinality $v$, the edge and face cardinalities $e=(v-2) 3$, $f=(v-2) 2$ 
are determined by Gauss-Bonnet and the Dehn-Sommerville relation $2e=3f$ holding for $2$-graphs.
The number of curvature-$(1/6)$ vertices has to be in the set $\{0,2,4,6,8,10,12\}$
because the Euler handshake formula $2e=\sum_{x \in V} {\rm deg}(x)$ implies that the total 
vertex degree is even and curvature $1/6$ vertices are the only odd degree vertices possible in a positive 
curvature $2$-graph. Note that in all cases, except for the case $v=12$,
two unit discs centered around a curvature $1/6$ vertex always intersect. 
(This holds simply by looking at the total cardinality as such a wheel graph has 6 vertices and two
disjoint discs have 12 vertices. It is only in the icosahedron case that we have two disjoint
disks with curvature $1/6$. That restricts the possibilities. 
For $v=6$, and $v=12$, we have platonic solids. For $v=7,v=8$ and $v=12,v=11$, we must have two adjacent 
vertices with different curvature. In the case $v=9$, the three degree $5$ vertices have to be 
adjacent. That fixes the structure. The vertex cardinality 11 is missing.
There is no positive curvature graph with $11$ vertices because such a graph has exactly one degree-
$4$ vertex and otherwise only degree-$5$ vertices. Look at the point $x$ with cardinality $4$ and then look 
at the spheres $S_r(x)$ around it: the sphere $S_2(x)$ of radius $2$ must have $4$ entries 
and so define an other degree-$4$ vertex.
\end{proof}

\paragraph{}
When this list was was first compiled on July 2 of 2011,  it used the built-in 
polyhedral graph libraries of degree $6-12$ in Mathematica 8. For vertex size smaller 
or equal to $9$, we then searched with brute force over $215$ suitable adjacency matrices and
then cross-referenced all two-dimensional ones for graph isomorphism.

\paragraph{}
There are always finitely many $2$-graphs with fixed vertex cardinality and so finitely
many also with non-negative curvature. How many are there? We believe there are none. 
The number can grow maximally polynomially in 
$r$ as we must chose $12$ vertices with curvature $1/6$ from $n$ or $10$ vertices of 
curvature $1/6$ and $1$ curvature $1/3$ etc or then 6 curvature 1/6 vertices.
While there are infinitely many fullerene type graphs 
with non-negative curvature, we also have only finitely many graphs of non-negative curvature which have 
no flat disc of fixed radius $r$. Also here, we do not have even estimates about the number of such graphs
depending on $r$. We only know that it must grow polynomially in $r$ because $12 \pi r^2 \geq v$ and the
number of positive curvature graphs with vertex cardinality $v$ is polynomial in $v$. 
We will in the last section comment on the case when the scalar curvature for $2$-graphs is replaced
by the Ricci curvature, which is a curvature on edges. Even so Ricci curvature does not satisfy a 
Gauss-Bonnet formula, there are only finitely many Ricci positive curvature 2-graphs but we do not know
how many there are. 

\paragraph{}
From the list of six positive curvature graphs, we see that all positive curvature graphs have 
diameter $2$ or $3$. Two of them have diameter $3$, the icosahedron with $12$ vertices, as well as 
the graph with $10$ vertices. 

\begin{figure}[!htpb]
\scalebox{0.4}{\includegraphics{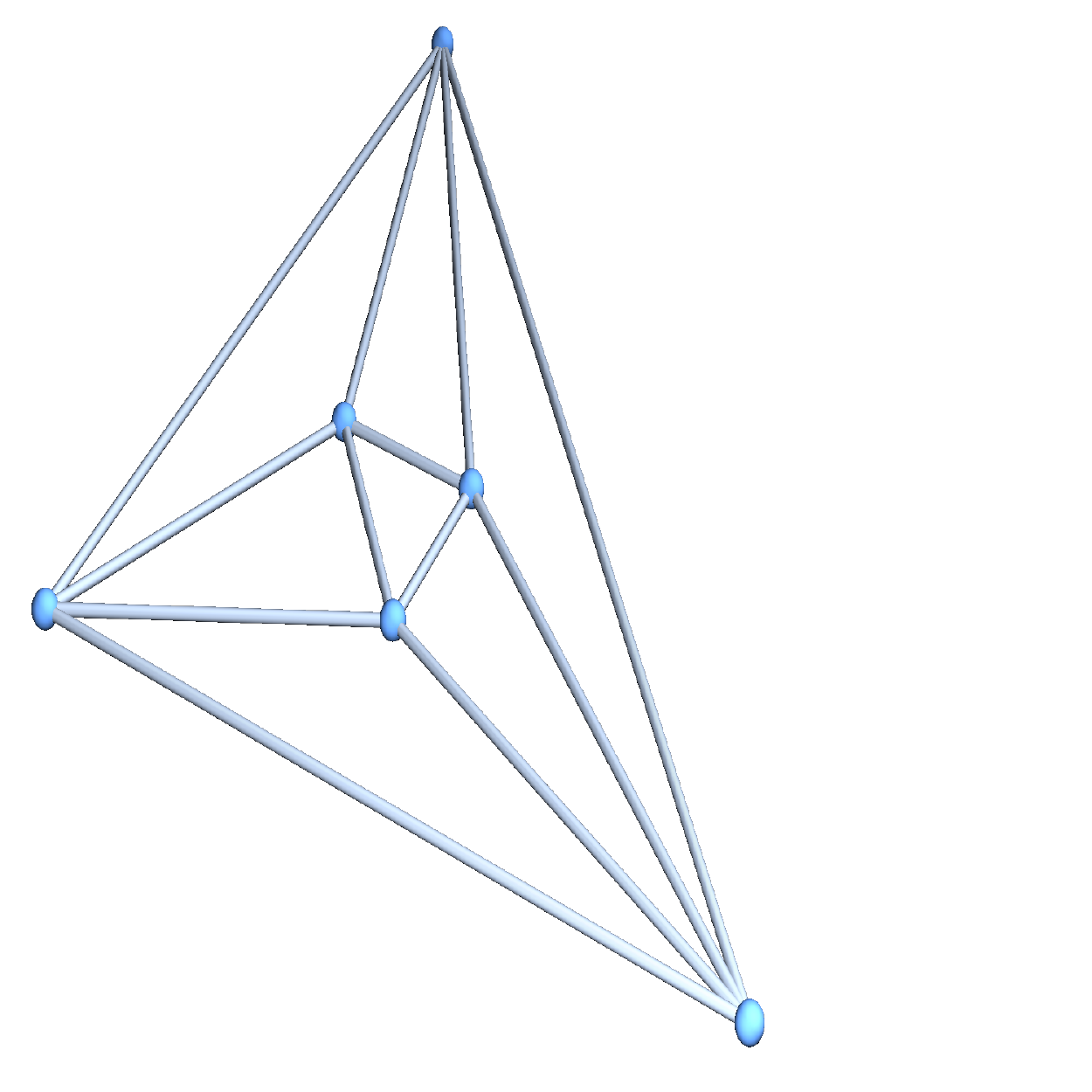}}
\scalebox{0.4}{\includegraphics{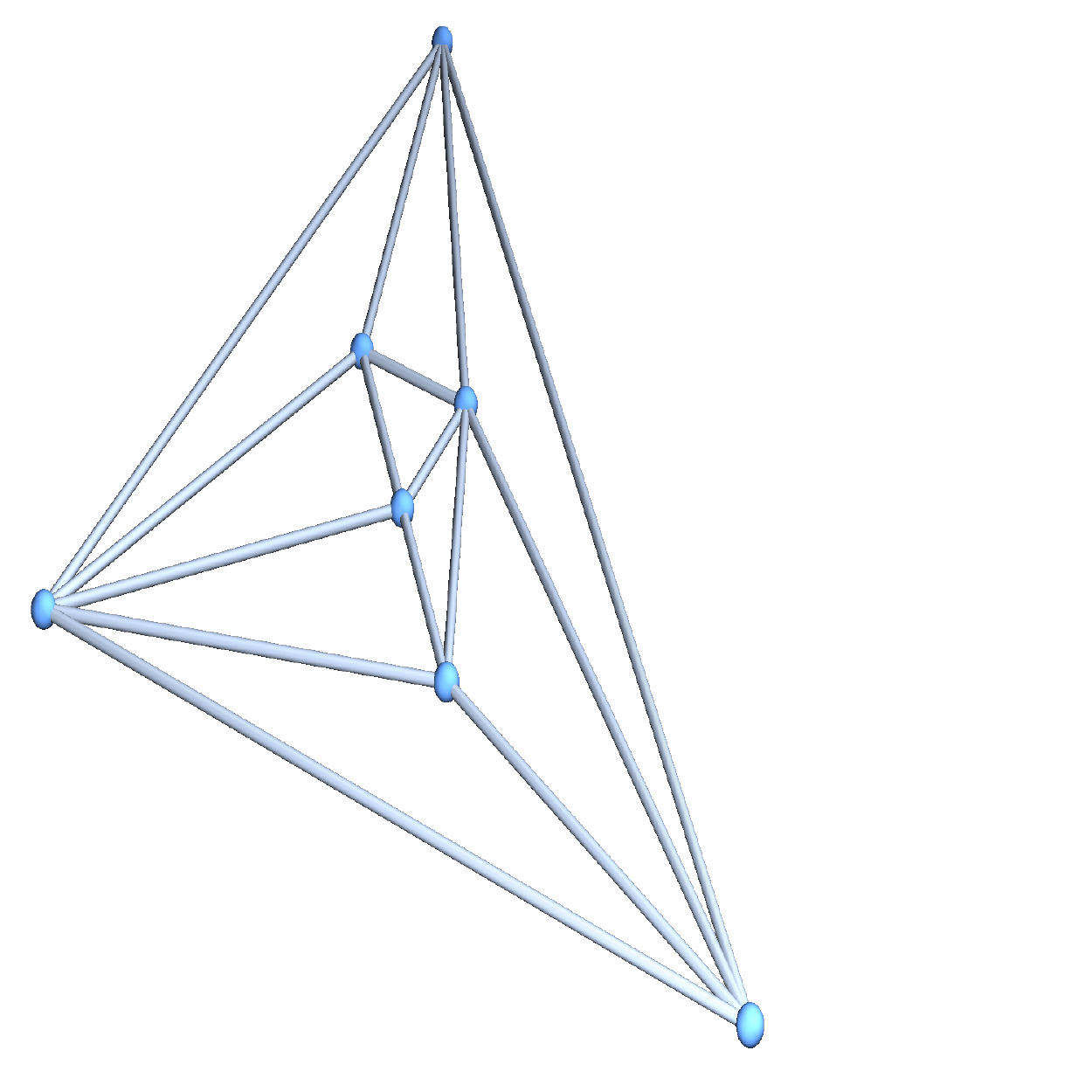}}
\scalebox{0.4}{\includegraphics{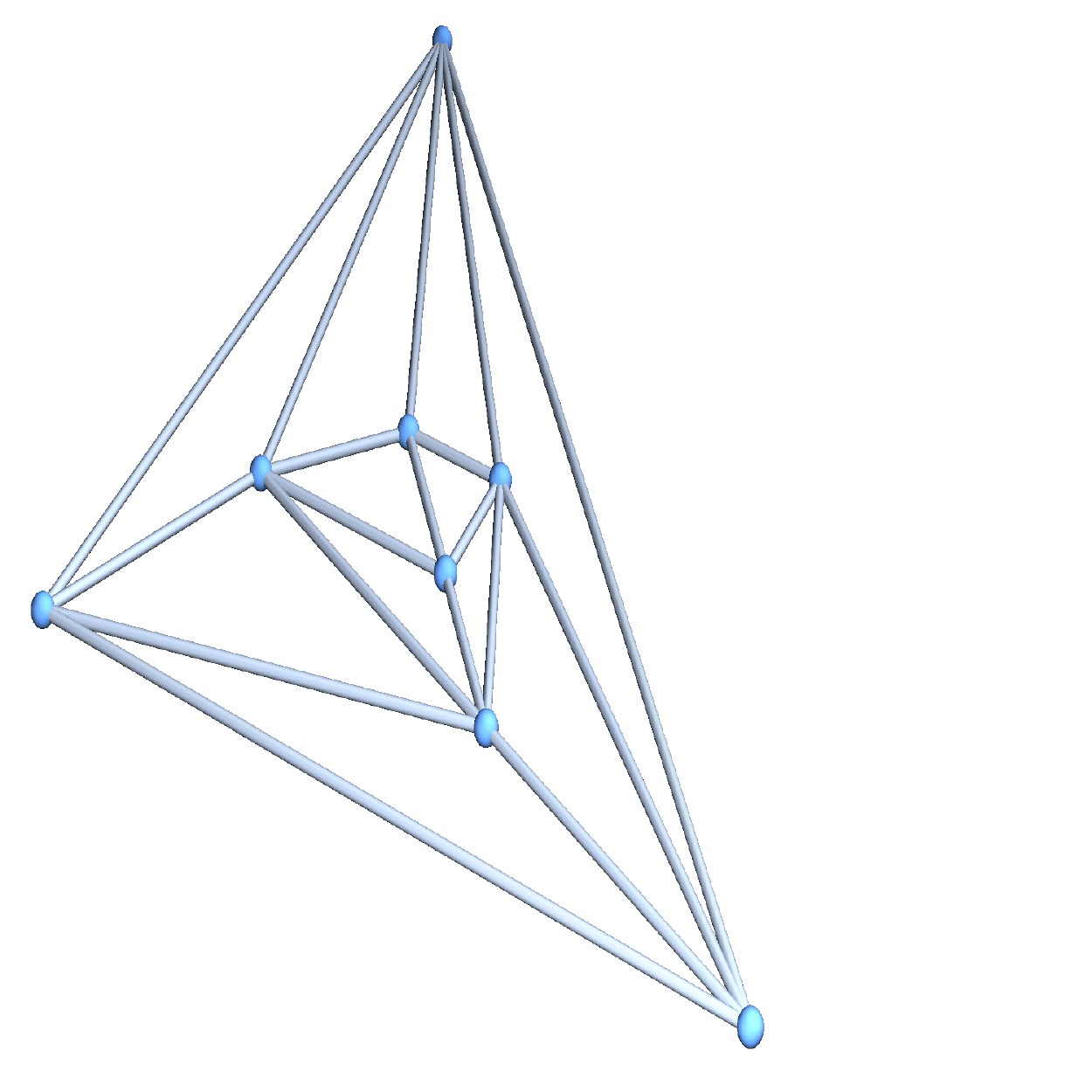}}
\scalebox{0.4}{\includegraphics{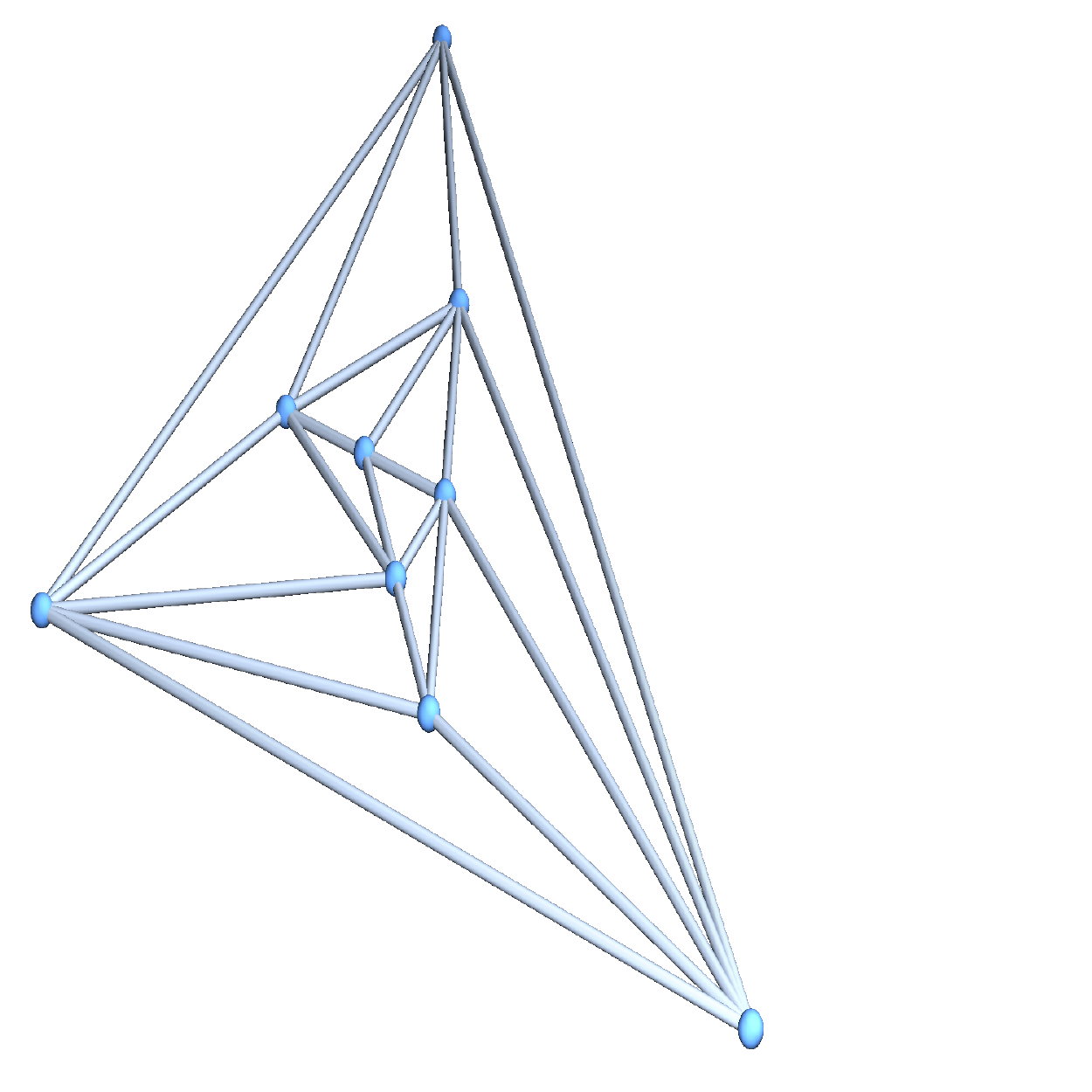}}
\scalebox{0.4}{\includegraphics{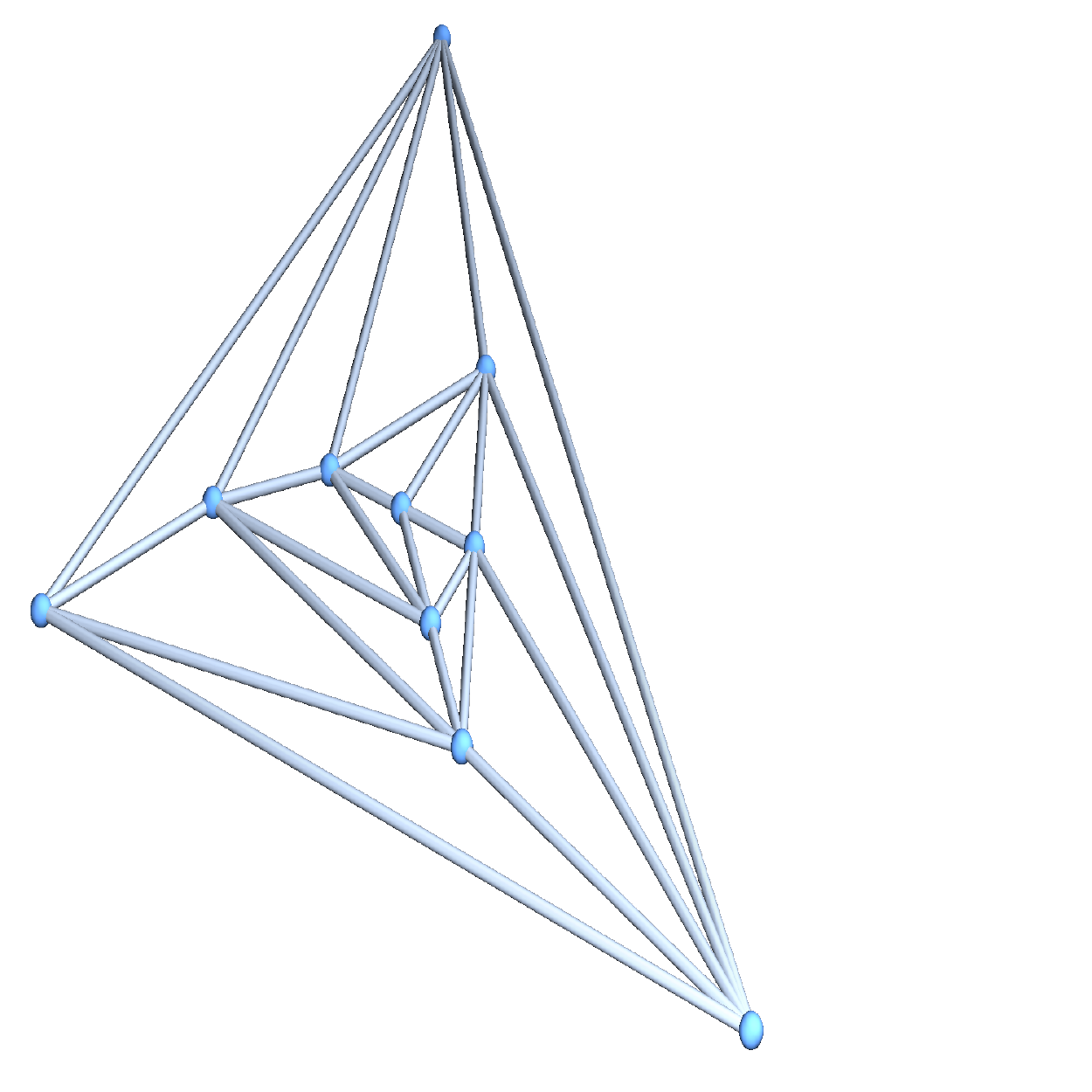}}
\scalebox{0.4}{\includegraphics{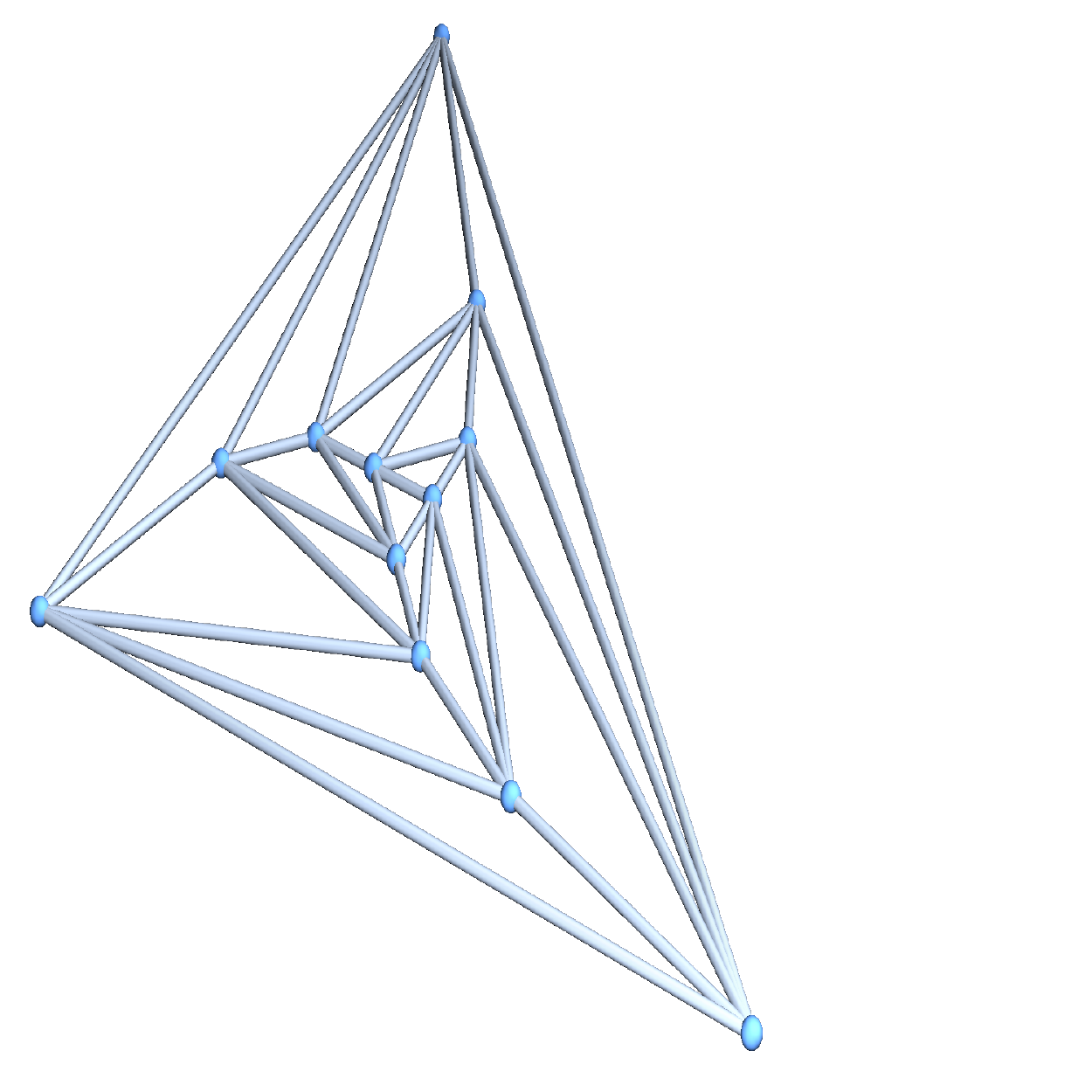}}
\label{2-graphs}
\caption{
The $6$ positive curvature $2$-graphs of dimension $2$. These are ``six
little mice".
}
\end{figure}

\begin{figure}[!htpb]
\scalebox{0.4}{\includegraphics{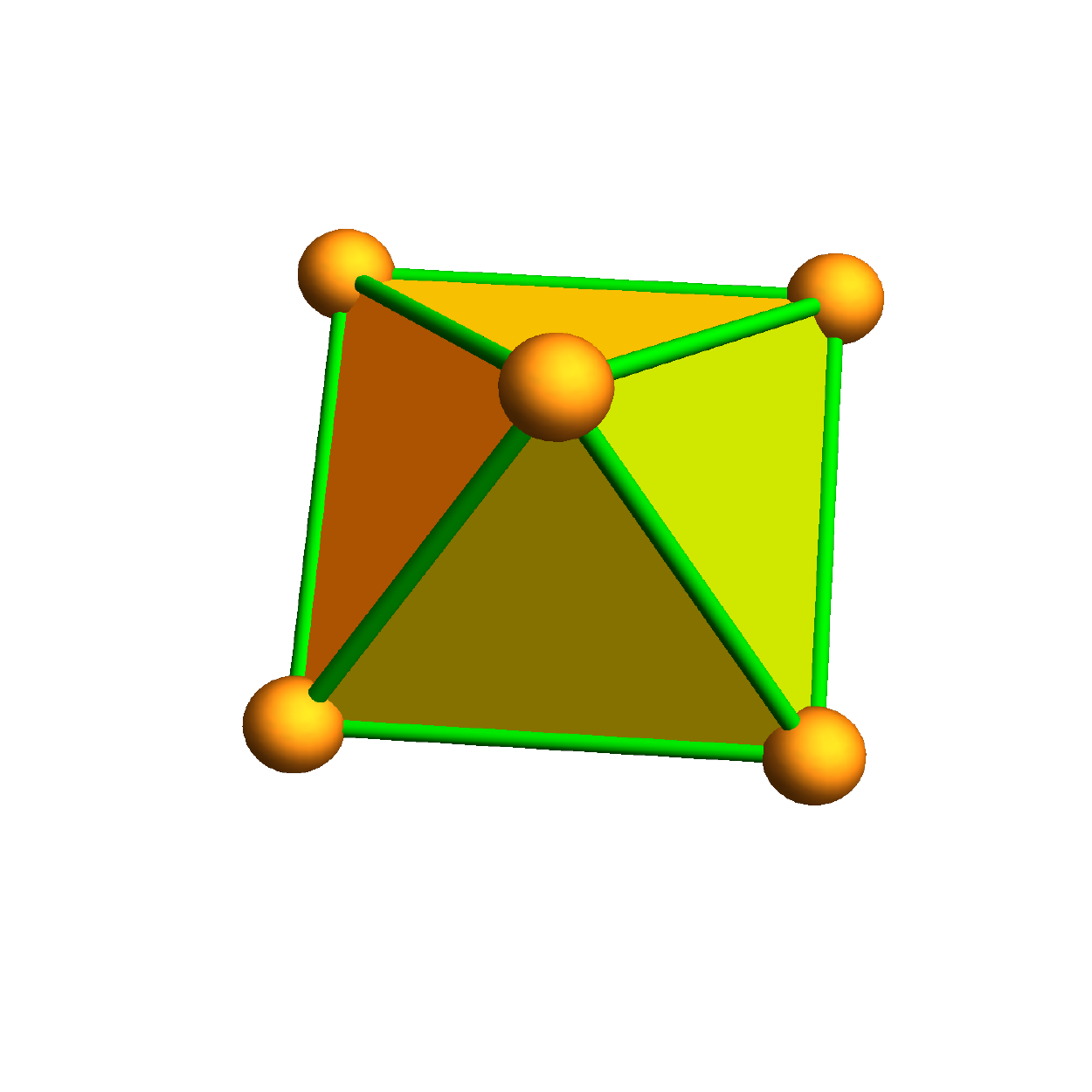}}
\scalebox{0.4}{\includegraphics{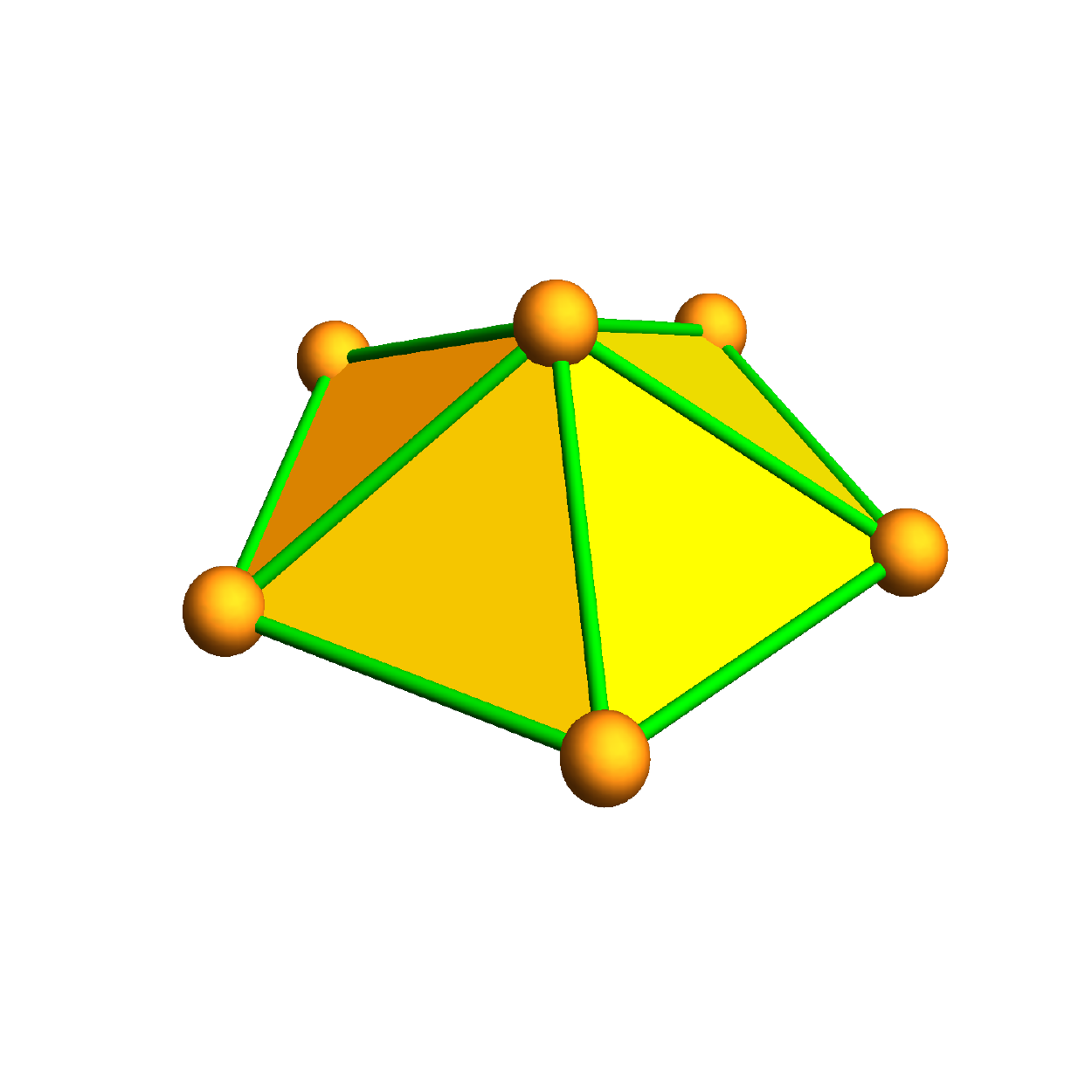}}
\scalebox{0.4}{\includegraphics{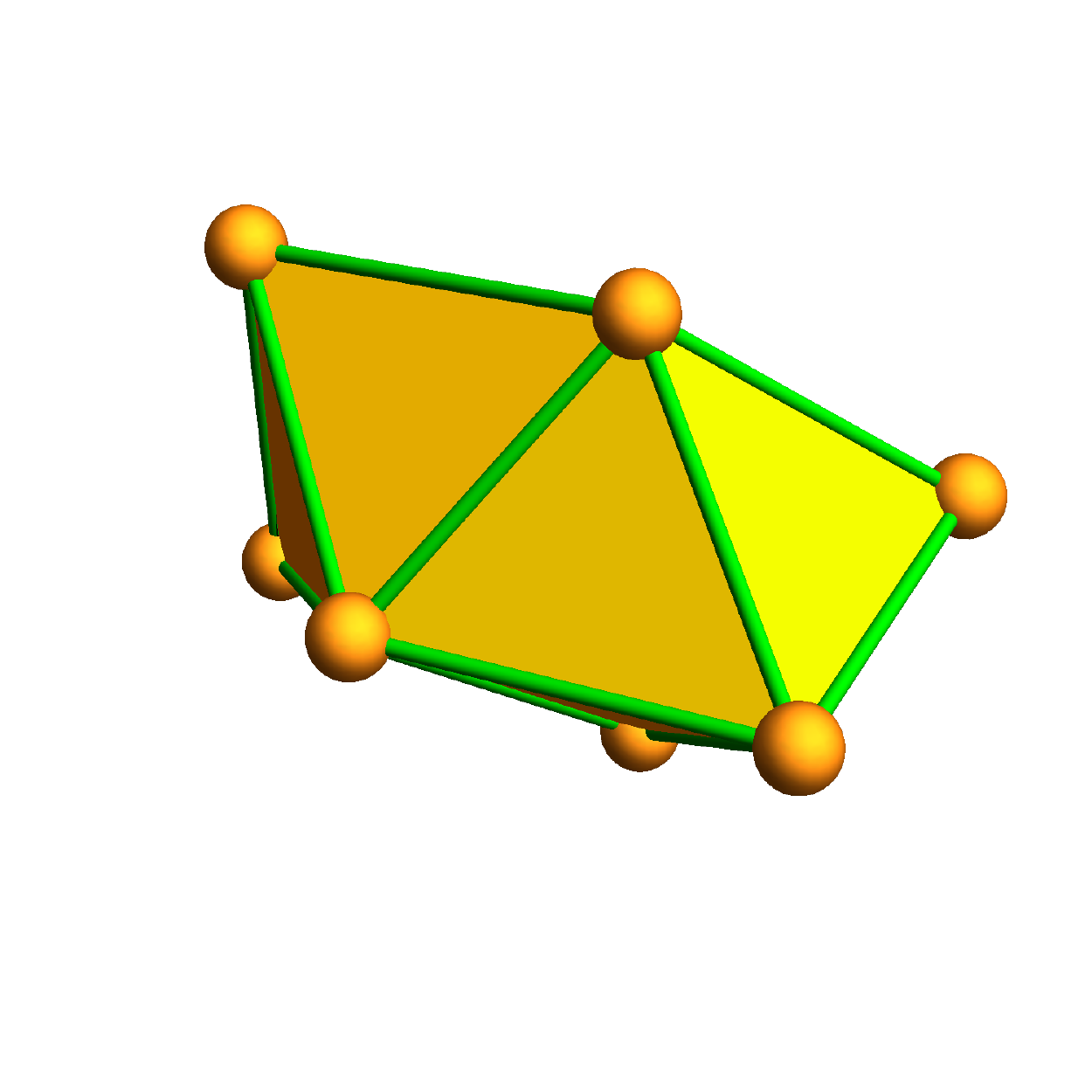}}
\scalebox{0.4}{\includegraphics{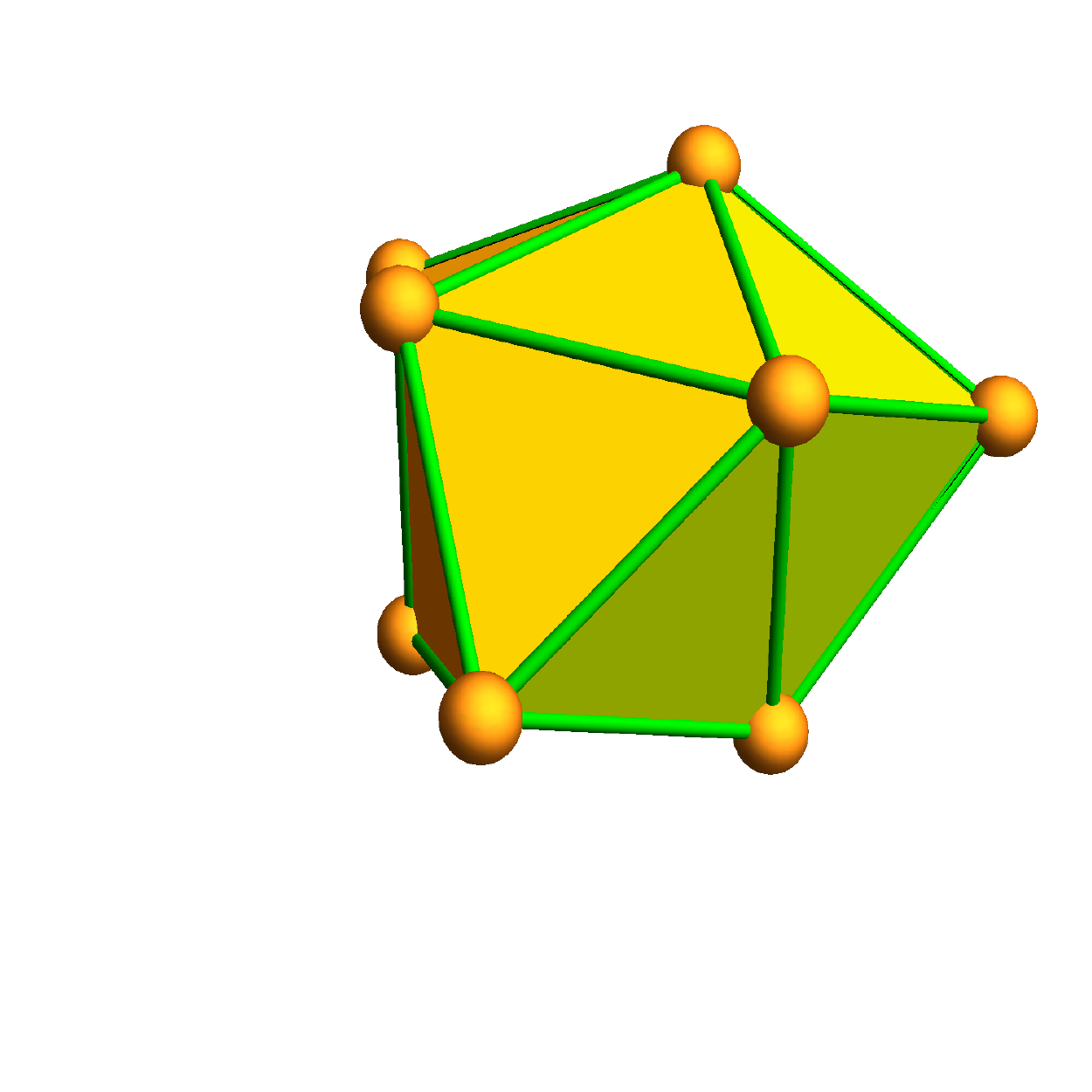}}
\scalebox{0.4}{\includegraphics{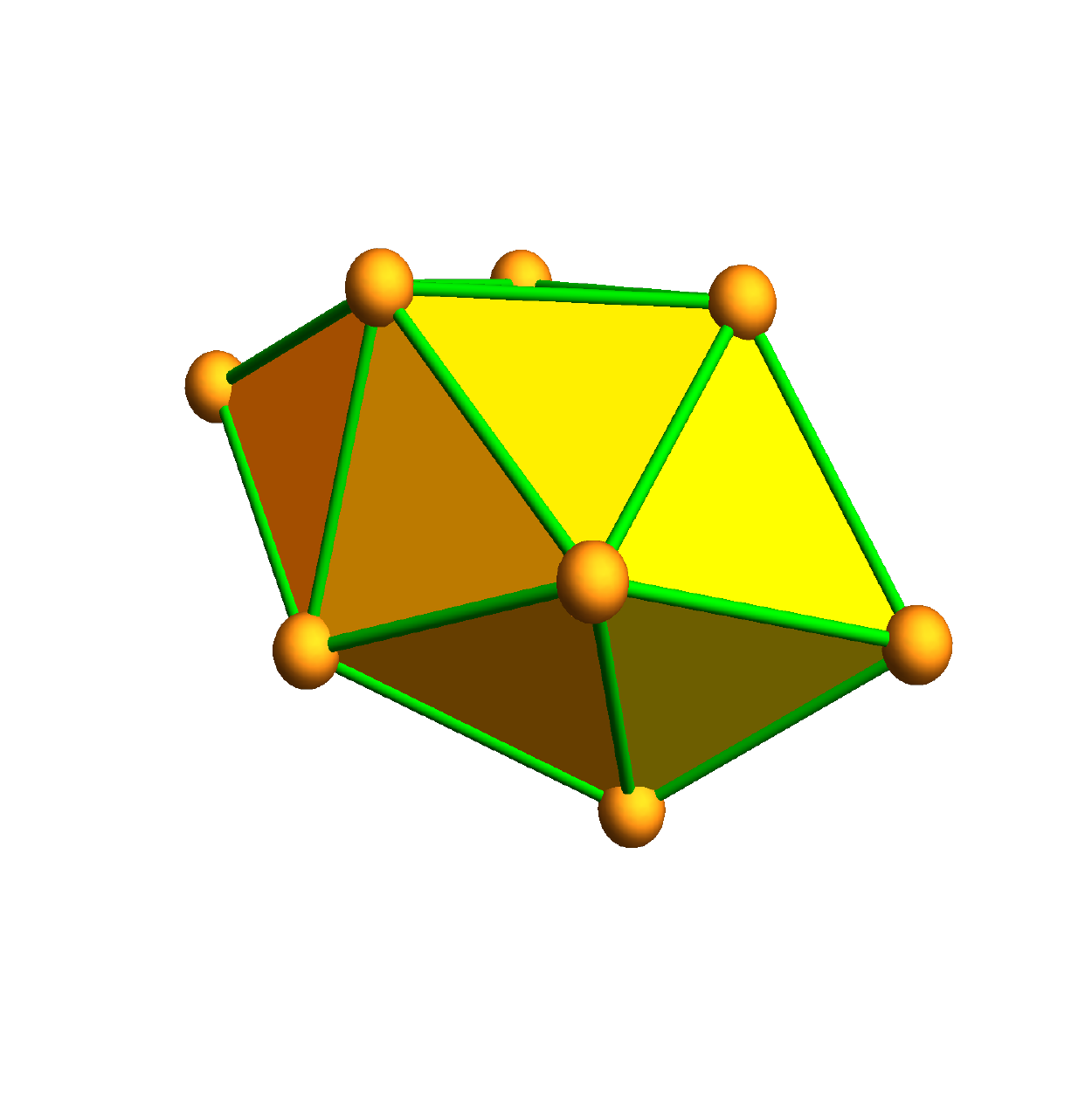}}
\scalebox{0.4}{\includegraphics{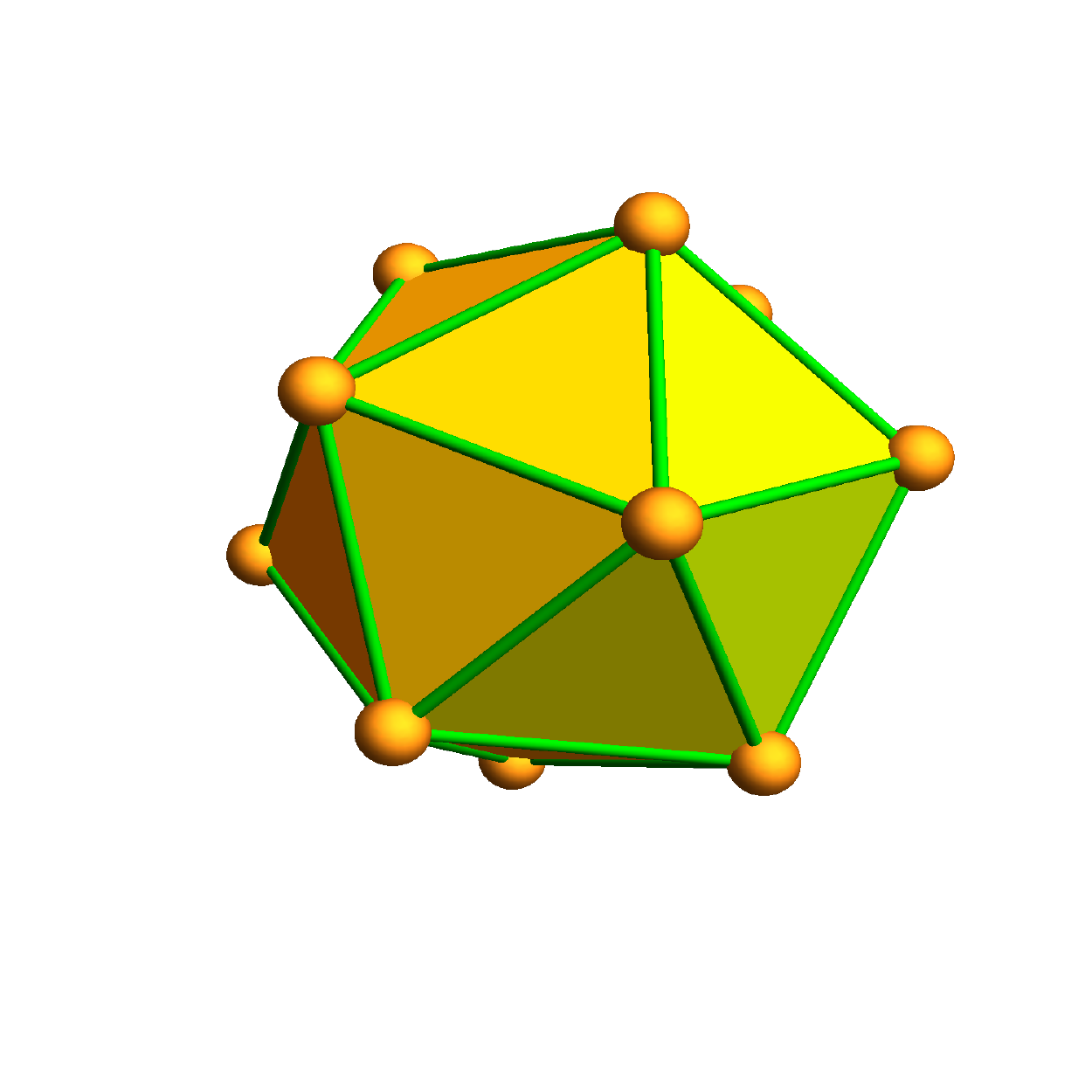}}
\label{2-graphs3D}
\caption{
The same six graphs are displayed when embedded in space and realized as polyhedra. 
Note however that the simple sphere theorem is a combinatorial result which 
does not rely on any geometric realization. Unlike in discrete
differential geometry frame works like Regge calculus, we do not care about
angles, lengths or other Euclidean notions. 
}
\end{figure}

\begin{figure}[!htpb]
\scalebox{0.9}{\includegraphics{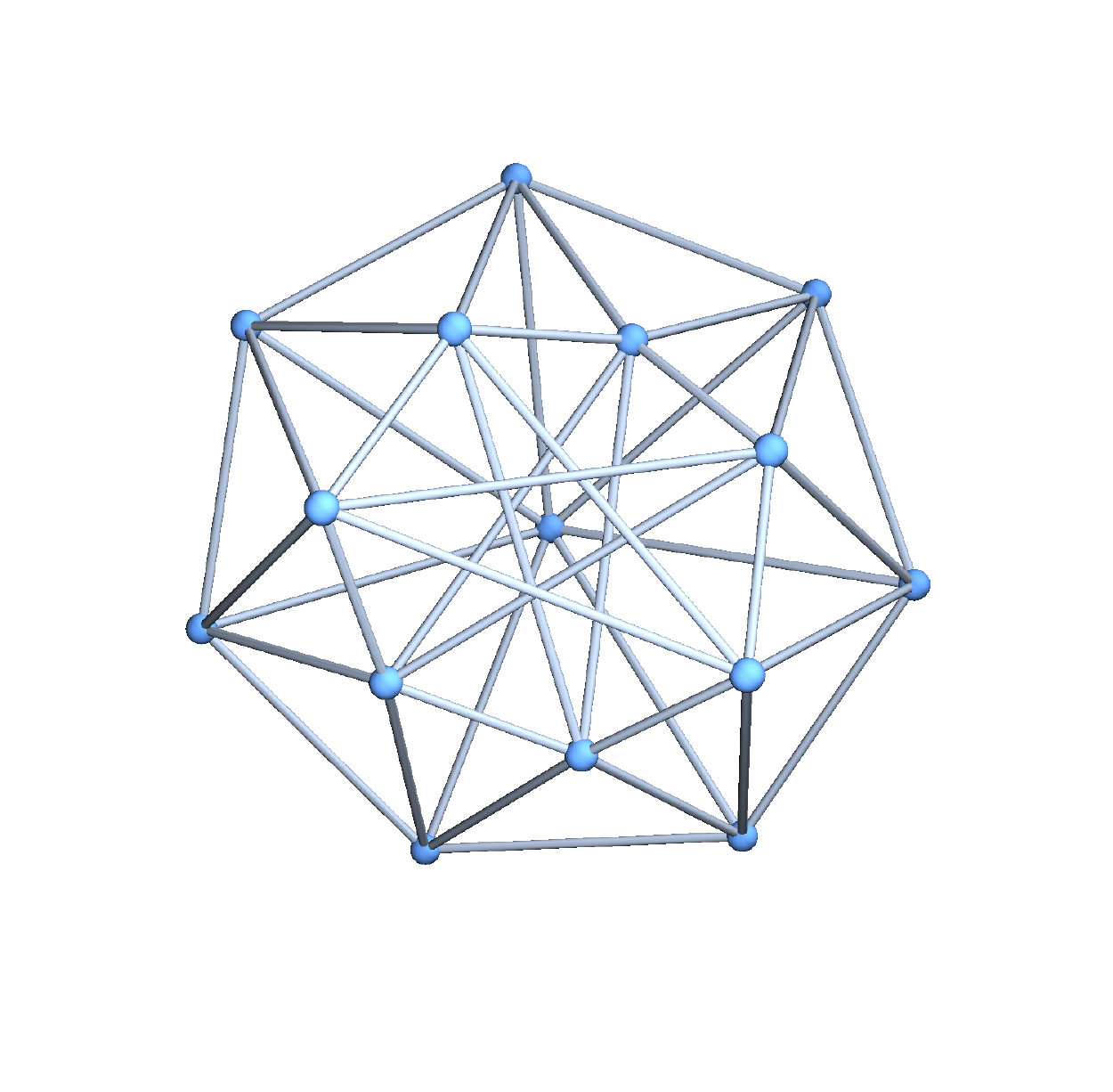}}
\label{jenny}
\caption{
This figure shows a projective plane $G$ with $15$ vertices. I learned about
this particular graph from Jenny Nitishinskaya, who constructed this graph 
as an example of a $2$-graph with chromatic number 5. This chromatic number
property is related to the fact that $G$ can not be the boundary of a simply 
connected $3$-graph \cite{knillgraphcoloring}. 
The curvature of $G$ is negative at one point. There is 
no room to implement a projective plane as a positive curvature graph: as 
the Euler characteristic of a projective plane is $1$ and the curvature is at 
last $1/6$ at every point, there can only be $6$ vertices. 
This means that the graph is a wheel graph with $5$ boundary points. 
The non-existence of positive curvature $2$-graphs which implement 
projective planes makes the orientability condition unnecessary in the 
simple sphere theorem. 
}
\end{figure}

\section{Geomag lemma}

\paragraph{}
The proof of the simple sphere theorem is elementary and constructive in any dimension. 
Given a positive curvature graph $G$, we will show that why after removing a vertex $x$, 
we have a contractible graph $G-x$. 
The reason is that if we take away a unit ball, the remaining graph is small of radius $3$ so that it 
can be covered by a ball of radius 2, for which every boundary point is covered with 
a ball of radius 1. To say it in other words, one can see that $G$ is a union of two balls.
This is known in the continuum: a $n$ dimensional smooth complete manifold which is the union of two
balls is homeomorphic to a sphere \cite{Petersen}. 
This characterization of spheres also holds in the discrete.

\paragraph{}
A graph is called a $d$-ball, if it is of the form $G-x$, where $G$ is a $d$-sphere. 
By definition, the boundary of a $d$-ball (the set of vertices $y$ where $S(y)$ is not
a $(d-1)$-sphere), is a $(d-1)$-sphere because it agrees with $S(x)$ of the original 
$d$-sphere $G$. By induction in dimension the unit sphere $S(y)$ of a boundary point $y$ 
is a $(d-1)$-ball because $S(y)+x$ is the unit sphere $S_G(y)$ in $G$. 

\begin{lemma}
A $d$-graph $G$ which is the union of two $d$-balls is a $d$-sphere.
\end{lemma} 

\begin{proof}
Let $G=A \cup B$, where $A,B$ are balls. Take a vertex $x$ at the boundary of $B$. 
By definition, as $B$ is contractible, we can remove vertices of $B$ until
none is available any more and we have only vertices of $A$ left. Now we again
by definition can remove vertices of $A$, until nothing is left. 
\end{proof} 

\paragraph{}
This is a Lusternik-Schnirelmann type result \cite{josellisknill}. 
The {\bf Lusternik-Schnirelmann category or simply category} ${\rm cat}(G)$
of a graph $G$ is the minimal number of contractible graphs needed
to cover $G$. Define ${\rm crit}(G)$ as the minimal number of critical points which an
injective function $f$ can have on $G$. The Lusternik-Schnirelmann theorem is
${\rm cat}(G) \leq {\rm crit}(G)$. Especially, $d$-spheres can be characterized as
as the $d$-graphs with category $2$. Related is the
Reeb sphere theorem which tells that $d$-graphs which admit a function with 
two critical points is a sphere \cite{knillreeb}. This implies the 2-ball theorem. 

\paragraph{}
For a $d$-graph $G$, and two vertices $x,y$, a curve $C$ connecting $x$ with $y$ is called
a {\bf geodesic arc} if there is no shorter curve in $G$ connecting $x$ and $y$. A {\bf closed curve}
$C$ in $G$ is called a {\bf geodesic loop} if for any two vertices $x,y$ in $C$, there is a 
geodesic arc from $x,y$ which is contained in $C$. 

\paragraph{}
The key is the following {\bf geomag lemma}: 

\begin{lemma}[Geomag]
a) Given any $2$-dimensional surface $S$ with boundary embedded in a $2$-graph $G$ and a point $x$ 
on the boundary $S$, there exists a wheel graph centered at $x$ which extends the surface.
b) Given a geodesic arc $C$ from $x$ to $y$, there exists a two-dimensional embedded 
surface $S$ (a 2-graph with boundary) which contains $C$.  \\
\end{lemma}

\begin{proof}
a) Let $x$ be a boundary point of the surface $S$, (a boundary point is a point where $S(x) \cap S$ 
is an arc on the $(d-1)$-sphere $S(x)$ and not a circle). 
We can now build a geodesic on $S(x)$ connecting the arc $S(x) \cap S$ but which is disjoint 
from the arc $AB$ (see the lemma below). 
This circle completion extends the surface $S$ to a larger surface by including the ``magnets"
from the arc $S(x) \cap S$. \\
b) Start with the boundary point $x$ of the geodesic arc $C$ and build a wheel graph $H$ centered at 
$x$. Now extend the surface as in part $a)$ at $H \cap C$. Continue extending the surface until 
reaching $y$. Now, we have a two-dimensional surface $S$ with boundary.  
\end{proof}

\paragraph{}
Here is an other lemma which is inductively used to extend a surface. 

\begin{lemma}
Given two vertices $x,y$ in a $d$-sphere $G$ and a geodesic arc $xy$, 
then there this arc can be extended to a circle $C$ in $G$ (a circle in $G$ 
is an embedded $1$-graph in $G$).
\end{lemma}
\begin{proof}
We remove a distance $1$ neighborhood $N$ of non-boundary points of $xy$ in $G$. 
The graph $N$ can be constructed as the union of all unit-balls centered at vertices
of $xy$ different from $\{x,y\}$. Now just take a geodesic in the graph $G-N$. 
It completes the graph and does not touch $xy$ in any place different from 
$\{x,y\}$. 
\end{proof}

\paragraph{}
The example of the octahedron graph $G$ and two an arc connecting two antipodal
vertices $x,y$ in $G$ is a situation, where the circle completion is unique. 

\paragraph{}
A consequence is:

\begin{coro}[Loop extension]
Any geodesic loop is part of a $2$-graph $S$ which is
immersed in $G$.
\end{coro}

\paragraph{}
More importantly, we have Bonnet-Myers theorem type result which bounds the diameter
of positive curvature $d$-graphs: 

\begin{coro}
The diameter of any positive curvature graph is $\leq 3$.
\end{coro}

\paragraph{}
There is no obvious analogue of the classical Cheng rigidity theorem 
characterizing positive curvature manifolds with maximal diameter as ``round spheres".
In the graph case, a maximal diameter $3$ positive curvature graph can already 
in dimension $2$ lead to different graphs. There are exactly two positive curvature 
$2$-graphs with maximal diameter $3$. And they are not-isometric $2$-spheres. 

\paragraph{}
A $d$-graph is {\bf simply connected} if every closed path $C$ in $G$ can be deformed to a point.
A {\bf deformation step replaces} two edges in $C$ contained in a triangle with the third edge in the
triangle ($2$-simplex) or then does the reverse, replaces an edge with the complement of a triangle.
This is equivalent to the continuum. We can define an addition of equivalence classes of
closed curves and get $\pi_1(G)$, the fundamental group. It is the same graph as when looking at the
classical fundamental group of a geometric realization but we do not need a geometric realization). 
A $d$-graph is simply connected, if the fundamental group is the trivial group. 

\begin{coro}[Synge]
Every $d$-graph of positive curvature is simply connected.
\end{coro}
\begin{proof}
A geodesic loop in $G$ can be extended to a 2-sphere $S$. This $2$-sphere $S$
is simply connected. We can deform the loop to a point by making the 
deformation on $S \subset G$. 
\end{proof}

\section{Classical results}

\paragraph{}
In this section, we mention some history. The combinatorial version can help to understand a
major core point of differential geometry: ``local conditions like positive curvature can have
a global topological effect". A special question is to see how positive curvature relates to the
diameter or injectivity radius, an other question is how it affects cohomology. 
The ultimate question is to relate it to a particular class of manifolds like 
spherical space forms (quotients of a sphere by a finite subgroup of the orthogonal group) 
or spheres and in which category the relation is done (i.e. continuous or diffeomorphism).
For a panoramic view over some major ideas of differential geometry, see \cite{BergerPanorama}. 
For history, see \cite{Berger2002}. 

\paragraph{}
The topic of search for local conditions enforce global conditions is central in differential geometry.
To cite \cite{Berger2002}: {\it ``since Heinz Hopf in the late 20's the topic of curvature and 
topology has been and still remains the strongest incentive for research in Riemannian geometry".}
The theme is that positive curvature produces some sort of ``sphere" and that
negative curvature produces spaces which have universal covers which are Euclidean spaces. 
The former are sphere theorems pioneered by Hopf and Rauch, the later
Hadamard-Cartan type results. In dimension $1$, where curvature assumptions are mute, one has both,
a compact and connected 1-manifold is always a circle, which is a sphere and the universal cover is the real line.
In the continuum, we need an orientability assumption to get Synge or a pinching condition to get a sphere theorem.
Local-global statements appear also in combinatorics: the $4$-color theorem is a global statement about
the maximal number of colors if the local injectivity condition is satisfied. The $4$-color theorem is 
equivalent to the statement that $2$-spheres have chromatic number $3$ or $4$. 

\paragraph{}
Synge's theorem of 1936 is one of the oldest general results about positive curvature
Riemannian manifolds \cite{Synge,docarmo94}. It already used a general bound on the diameter
$L> \pi/\sqrt{K}$ which is called Myer's theorem in terms of the minimal value $K$ of the curvature. 
Synge's theorem states that a compact orientable and connected positive curvature manifold
is simply connected. It is not as deep as the sphere theorem of Rauch-Berger-Klingenberg 
\cite{Rauch51}, Berger and Klingenberg (see \cite{BergerPanorama, Petersen}),
which assures that a sufficiently pinched orientable positive 
curvature manifold is a sphere. More recent are differentiable sphere theorems, in particular
the theorem of Brendle and Schoen \cite{BrendleSchoen} which assures that
a complete, simply connected, quarter-pinched Riemannian manifold is diffeomorphic to the
standard sphere. By Synge theorem, one can replace the simply connectedness assumption 
with orientability.

\paragraph{}
Synge's result \cite{Synge} telling that a positive curvature manifold is simply connected
is appealing as it mixes positive curvature, a local differential geometric notion with
orientability and simply connectedness which are both of a global and topological nature. 
Synge already uses isometry argument. In \cite{AkhilSynge}, Synge's theorem is proven using a
theorem of Weinstein which tells that for an even-dimensional positive curvature manifold, an 
orientation preserving isometry $f: M \to M$ has a fixed point. 
For an odd-dimensional positive curvature manifold, an orientation reversing 
orientation has a fixed point. Weinstein's theorem follows for spheres or projective spaces
from the Lefschetz fixed point theorem but it is more general and is also remarkable as it only uses
the positive curvature assumption. Synge's theorem uses calculus of variations: a minimal geodesic connecting
$p$ with $f(p)$ has as a second variation a linear operator which is positive definite and
an isometry of a compact oriented even dimensional manifold has a fixed point.

\paragraph{}
As for the beginnings of the sphere theorem, (\cite{BrendleSchoen} and the introduction to 
\cite{Brendle2010, Schoen2017} give overviews), it was Heinz Hopf who, motivated 
largely by physics, conjectured, starting in 1932 that a sufficiently pinched positive curvature 
space must be a sphere. After Rauch visited ETH in 1948/1949, he proved the first theorem 
assuming a pinching condition of a bout $3/4$. Rauch's theorem is remarkable as it
is the first of this kind. It introduced a ``purse string method" which can be seen as a continuous
version of a geomag argument. 
The topological sphere theorem of 1960, proven by M. Berger and W. Klingenberg proves that under the
optimal $1/4$ pinching condition an orientable positive curvature manifold has to be 
a sphere. In 2007, R. Schoen and S. Brendle got then the differential case, 
using newly available Ricci flow deformation methods. 

\section{Questions}

\paragraph{}
Can one use the geomag idea to design a proof of the classical Synge theorem by some sort
of approximation? It would require some more technical things. Start with a geodesic
two-dimensional surface patch and extend it to a two-dimensional surface along a geodesic. 
Let $\gamma$ be a closed curve in $M$ which can not be contracted and let $p,q$ be two points in $\gamma$
of maximal distance apart. Extend $\gamma$ to a surface and
smooth it out, still making sure the positive curvature surface $S$ remains 
immersed in $M$. As it is a connected, compact
two-dimensional Riemannian manifold, it must be the projective plane or the $2$-sphere. In the case
when $M$ is orientable, $S$ is orientable too and must be a sphere. One can now contract $\gamma$ in 
the contractible 2-dimensional surface $S$. 

\paragraph{}
In order to weaken positive curvature to get closer to the continuum sphere theorems, one could
use the notion of {\bf Forman-Ricci curvature} \cite{BHJLW}. Lets just call it Ricci curvature. 
Ricci curvature of a $2$-graph is a function attached to edges. 
Positive Forman curvature means that for
every edge $e=(a,b)$, the Ricci curvature $K(e)=1-{\rm deg}(a,b)/6$ is positive, 
where ${\rm deg}(a,b) = ({\rm deg}(a) + {\rm deg}(b))/2$.
The Ricci curvature does not satisfy any Gauss-Bonnet formula but it is likely to lead to a 
sphere theorem. We have not counted the number of $2$-graphs with that positive Ricci curvature.

\paragraph{}
By averaging curvatures differently, one can get other notions of curvature. Still according to 
Forman, one can assign a curvature to triangles $f=(a,bc) $ by $K(f) = 1-{\rm deg}(a,b,c)/6$ 
where ${\rm deg}(a,b,c)$ the average  of the vertex degrees of the vertices $a,b,c$. 
There is no doubt that there is a sequence of sphere theorems in a combinatorial setting 
which captures more aspects the continuum and still relies only on the geomag construction idea. 
This still needs to be done and there will be a threshold, where projective planes will be allowed
forcing to include the orientability condition as in the continuum. 

\begin{figure}[!htpb]
\scalebox{0.6}{\includegraphics{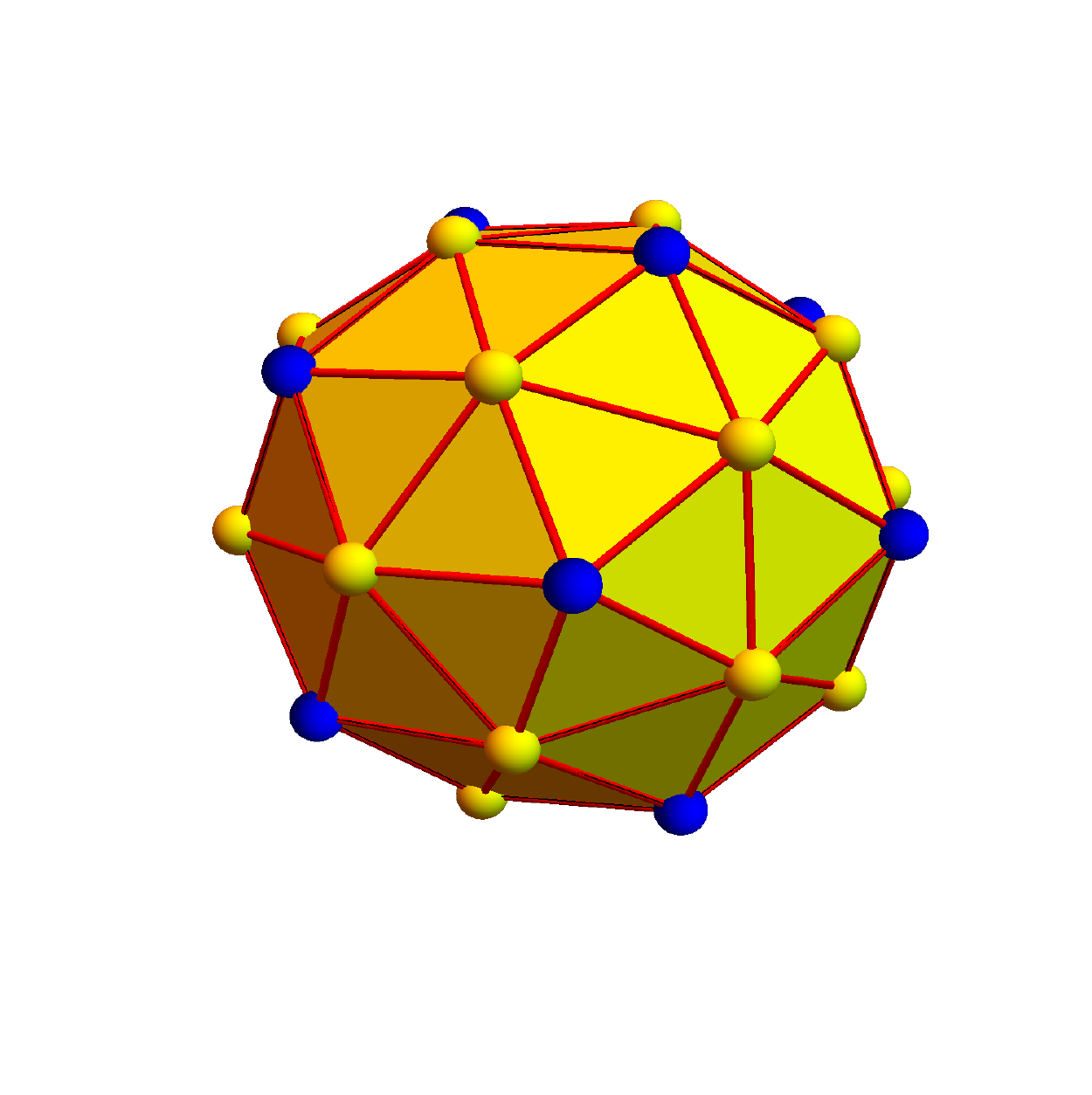}}
\scalebox{0.6}{\includegraphics{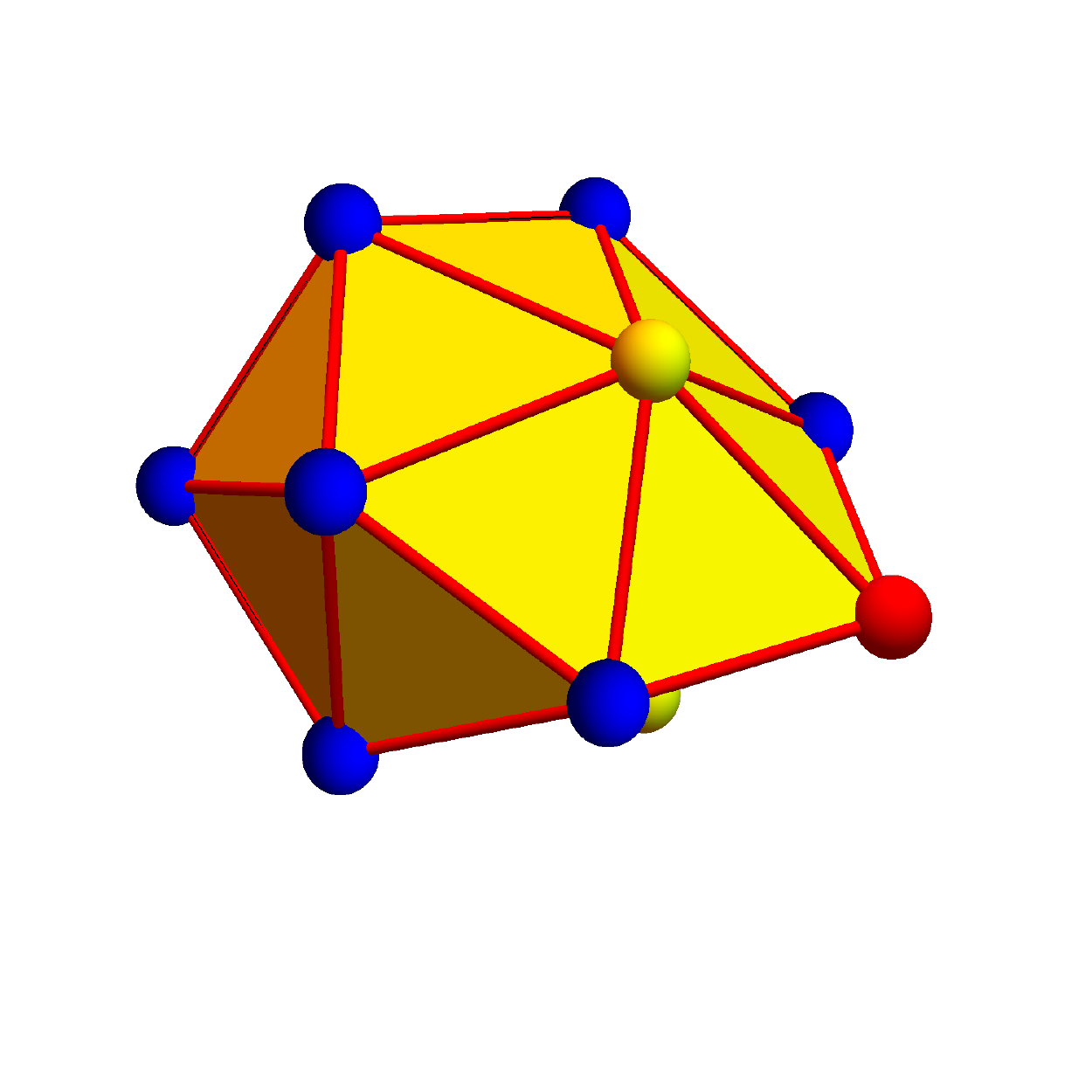}}
\label{jenny}
\caption{
The Pentakis dodecahedron has vanishing Ricci curvature at some edges.
A once edge refined icosahedron has positive Ricci curvature everywhere
but two vertices with zero curvature. Unlike the curvature $1-{\rm deg}(x)/6$ for
$2$-graphs, the Ricci curvature does not satisfy any Gauss-Bonnet identity.
}
\end{figure}

\paragraph{}
Various definitions of curvature have been proposed in the discrete. 
Some of them are quite involved. Looking at second order curvatures $S_2(x)-2S_1(x)$ 
does not work well and are computationally complicated. Already in the planar case
\cite{elemente11}, this Puiseux type curvature already leads to sensitive issues
when proving an Umlaufsatz. Ollivier type curvatures \cite{Ollivier,BHJLW} are even 
tougher to work with. We take the point of view that the definition of curvature 
should be simple and elementary. 

\paragraph{}
An interesting open combinatorial problem is to enumerate all positive curvature 
graphs in $d$, when taking the curvature assumption of this paper.
Similarly, we would like to get all
$2$-dimensional genus $g>1$ graphs of negative curvature (we believe that there are none). 
All $2$-dimensional positive curvature cases are obtained by doing edge refinements starting with the 
octahedron, We can ask whether every $d$-dimensional positive curvature graph 
can be obtained from edge refinements. 

\paragraph{}
One can define notions of Ricci curvatures different than what Forman did. The following definition
gives the scalar curvature in the case $d=2$. 
Can we prove as in the continuum that positive Ricci curvature, a quantity assigned to edges given as the 
average of all curvatures of wheel graphs containing $e$, imply a bound on the diameter of $G$? 
This would be a more realistic Myers theorem. 

\paragraph{}
What possible vertex cardinalities do occur of positive curvature graphs. 
In two dimensions, we see that for $d=2$, the icosahedron, the graph with a maximal number of vertices
sis unique. Is this true in dimension $d$ also? If not how many are there with maximal cardinality. 
We believe that the suspensions of icosahedron are the largest positive curvature graphs in any dimension. 
These are the cross polytopes of Schl\"afli. In the sense defined here, these cross polytopes are
the only Platonic $d$-spheres in dimension $d \geq 5$ (the other two classically considered, the 
hyper-cubes or the hyper-simplices are not d-graphs.)

\paragraph{}
In classical differential geometry dealing with $d$-manifolds $M$, there are a couple of 
notions of intrinsic curvature, curvature which does not depend on $M$ being 
embedded in a higher dimensional space. The theorema egregia of Gauss allows to see 
sectional curvatures as independent of the embedding and use it to define the 
Riemann curvature tensor, averaging sectional curvature over 2-planes intersecting in a line gives Ricci 
curvature, averaging all Ricci curvatures gives scalar curvature which enters the Hilbert action.
Then there is the Euler curvature, a Pfaffian of the Riemann curvature tensor appears in 
Gauss-Bonnet-Chern. 

\section*{Afterword}

\paragraph{}
The topic relating local properties like curvature with global topological features
is an interesting theme also in physics. The reason is that basic fundamental laws in physics are 
local by nature if information needs time to propagate. 
Curvature in particular is a fundamental local quantity. Relativity relates it to mass and energy. 
Having a definite sign of curvature is desirable for various reasons. Positive
curvature and orientability implies simply-connectedness. Negative curvature is often dubbed 
anti-de-Sitter and assures no conjugate points for the geodesic flow. This simplifies physics. 

\paragraph{}
The original investigations by Heinz Hopf have been motivated by physics. As 
cited in \cite{Schoen2017}, Hopf wrote in 1932:
``The problem of determining the global structure of a space form
from its local metric properties and the connected one of
metrizing - in the sense of differential geometry-a given topological
space, may be worthy of interest for physical reasons". At that time,
the geometrization of gravity due to Einstein had been 
a major drive to investigate more of differential geometry. 

\paragraph{}
Cosmological questions related to curvature about space-time
been investigated early, in particular by Willem de Sitter. 
A de Sitter space is a positive curvature analogue of the Minkowski space. 
The Synge result that it is simply connected is important as a non-simply connected
manifold produces twin paradox problems, where traveling along a geodesic
coming back to the same point in an equivalent reference frame produces serious
causality issues.

\begin{figure}[!htpb]
\scalebox{0.3}{\includegraphics{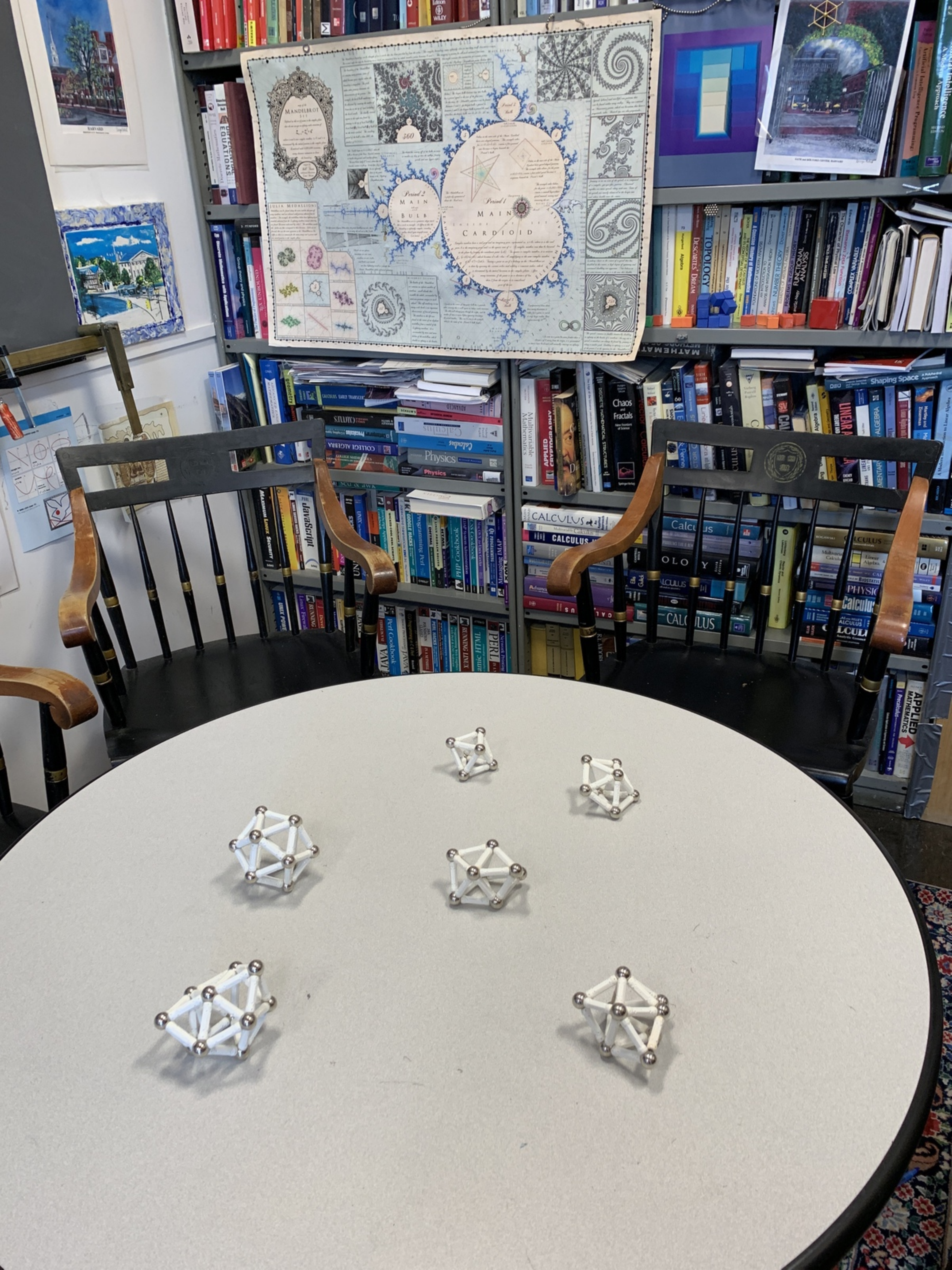}}
\label{jenny}
\caption{
The $6$ positive curvature $2$-graphs physically built 
with the magnetic building tool ``geomag".
}
\end{figure}

\bibliographystyle{plain}

\end{document}